\documentclass[11pt]{article}

\usepackage{latexsym,mathrsfs}
\usepackage{amsmath,amssymb}
\usepackage{amsthm,enumerate,verbatim}
\usepackage{amsfonts}
\usepackage{graphicx}
\usepackage{algorithm}
\usepackage{algorithmic}
\usepackage{url}
\usepackage{here}
\usepackage{geometry}
\usepackage{rotating}
\usepackage {tikz}
\usepackage{pgfkeys}
\usetikzlibrary{shapes,arrows}
\usepackage {subfigure}	
\usepackage {array}
\usepackage {geometry}         		
\usepackage{color}
\usepackage{xcolor}
\usepackage{cancel}
\usepackage{mcode}

\renewcommand{\algorithmicrequire}{\textbf{Input:}}
\renewcommand{\algorithmicensure}{\textbf{Output:}}

\DeclareMathOperator{\faces}{faces}

\newtheorem{lemma}{Lemma}
\newtheorem*{cor*}{Corollary}
\newtheorem{theorem}{Theorem}
\newtheorem{corollary}{Corollary}

\newtheorem*{definition*}{Definition}

\newtheorem{remark}{Remark}

\DeclareMathOperator{\rank}{rank}

\title{On the Linear Extension Complexity of Regular $n$-gons}

\date{} 

\author{Arnaud Vandaele\thanks{Department of Mathematics and Operational Research, 
Facult\'e Polytechnique, Universit\'e de Mons, Rue de Houdain~9, 7000 Mons, Belgium. 
Emails: 
 \texttt{\{arnaud.vandaele, nicolas.gillis\}@umons.ac.be.}} \and Nicolas Gillis$^*$  
\and Fran\c{c}ois Glineur\thanks{Universit\'e catholique de Louvain, CORE and ICTEAM Institute, 
B-1348 Louvain-la-Neuve, Belgium; \texttt{francois.glineur@uclouvain.be}.} 
}

\begin{document}

\maketitle

\begin{abstract} 
In this paper, we propose new lower and upper bounds on the linear extension complexity of regular $n$-gons. 
Our bounds are based on the equivalence between the computation of (i)~an extended formulation of size~$r$ of a polytope~$P$, and (ii)~a rank-$r$ nonnegative factorization of a slack matrix of the polytope~$P$.  
The lower bound is based on an improved bound for the rectangle covering number (also known as the boolean rank) of the slack matrix of 
the $n$-gons. 
The upper bound  is a slight improvement of the result of Fiorini, Rothvoss and Tiwary [Extended Formulations for Polygons, Discrete Comput.\@ Geom.\@ 48(3), pp.~658-668, 2012]. 
The difference with their result is twofold: 
(i)~our proof uses a purely algebraic argument while Fiorini et al.\@ used a geometric argument, and 
(ii)~we improve the base case allowing us to reduce their upper bound $2 \left\lceil  \log_2(n) \right\rceil$ by one when $2^{k-1} < n \leq 2^{k-1}+2^{k-2}$ for some integer $k$. 
We conjecture that this new upper bound is tight, which is suggested by numerical experiments for small $n$. Moreover, this improved upper bound allows us to close the gap with the best known lower bound for certain regular $n$-gons (namely, $9 \leq n \leq 13$ and $21 \leq n \leq 24$) hence allowing for the first time to determine their extension complexity. 
\end{abstract} 


\section{Introduction} \label{intro} 

An extended formulation (or extension) for a polytope $P$ is a higher dimensional polyhedron $Q$ 
such that there exists a linear map $\pi$ with $\pi(Q) = P$.  
The size of such an extended formulation is defined as the number of facets of the polyhedron $Q$. 
The size of the smallest possible extension of $P$ is called the (linear) extension complexity of $P$ and is denoted xc($P$). 
The quantity xc($P$) is of great importance since it characterizes the minimum information necessary to represent $P$. In particular, in combinatorial optimization, it characterizes the minimum size necessary to represent a problem as a linear programming problem (taking $P$ as the convex hull of the set of feasible solutions).  
Hence although $P$ might have exponentially many facets, $Q$ might have only a few, providing a way to solve linear programs over $P$ much more effectively. 
An example of such a polytope is the permutahedron, that is, the convex hull of all permutations of the set $\{1,2,\dots,n\}$ with $n!$ vertices and $2^n-2$ facet-defining inequalities, that can be represented as the projection of a polyhedron with $\mathcal{O}(n \log(n))$ facets~\cite{G09}. 

The characterization of the extension complexity has attracted much interest recently; in particular lower bounds since they provide provable limits of linear programming to solve combinatorial optimization problems; see, e.g., \cite{FMPTdW12}. 
For example, it was recently shown that the extension complexity of the matching polytope is exponential (in the number of vertices of the graph), answering a long-standing open question whether there exists a polynomial-size linear programming formulation for the matching problem~\cite{Ro14}  which implies that although it is solvable in polynomial time, 
the standard formulation cannot be written as a linear program with a polynomial number of inequalities. 

Interestingly, most lower bounds for the extension complexity of polytopes are based on a well-known linear algebra concept: the nonnegative rank. 
The nonnegative rank of a nonnegative $m$-by-$n$ matrix $M$, denoted $\rank_+(M)$, is the minimum $r$ such that there exist a nonnegative $m$-by-$r$ matrix $U$ and a nonnegative $r$-by-$n$ matrix $V$ such that $M = UV$. 
The pair $(U,V)$ is a rank-$r$ nonnegative fatorization of $M$. 
The link between the nonnegative rank and the extension complexity of a polytope, a seminal result of Yannakakis~\cite{Y91}, 
goes as follows. Let $P$ be a polytope in dimension $d$ with 
\begin{itemize}
\item $f$ facets expressed as linear inequalities $a_i^Tx \leq b_i$ $1 \leq i \leq f$, and 
\item $v$ vertices denoted $x_j \in \mathbb{R}^d$  $1 \leq j \leq v$. 
\end{itemize} 
The slack matrix $S_P \in \mathbb{R}^{f \times v}_+$ of $P$ is defined as 
\[
S_P(i,j) = b_i - a_i^T x_j \geq 0, \qquad \text{ for all } 1 \leq i \leq f, 1 \leq j \leq v. 
\]
Note that the slack matrix of a polytope is not unique since the inequalities can be scaled, and the rows and columns permuted but this does not influence its nonnegative rank; see~\cite{GGK13} for more details. 
Note also that $\rank(S_P) = d+1$ if $P$ is full dimensional.  
Then, we have 
\[
\rank_+ (S_P) \; = \; \text{xc}(P) . 
\]
Moreover any nonnegative factorization $(U,V) \geq 0$ of $S_P = UV$ provides an explicit extended formulation for $P$ (with some redundant equalities): 
\[  
P = \{ x \in \mathbb{R}^d \ | \ A x \le b \} = \{ x \in \mathbb{R}^d \ |\ A x + U y = b \text{ and } y \ge 0 \} , 
\]
where $A \in \mathbb{R}^{f \times d}$ with $A(i,:) = a_i$ for all $i$, 
and $b \in \mathbb{R}^{f}$. 
For example, the matrix 
\[
S_6 = \begin{pmatrix}
0&1&2&2&1&0\\
0&0&1&2&2&1\\
1&0&0&1&2&2\\
2&1&0&0&1&2\\
2&2&1&0&0&1\\
1&2&2&1&0&0
\end{pmatrix} 
\]
is a slack matrix of the regular hexagon (hence it has rank three) and has nonnegative rank equal to five: 
\[
S_6 = 
\begin{pmatrix}
0&0&0&1&2\\
0&1&0&0&1\\
0&1&1&0&0\\
0&0&2&1&0\\
1&0&1&0&0\\
1&0&0&0&1\end{pmatrix}
\begin{pmatrix}
1&2&1&0&0&0\\
0&0&0&1&2&1\\
1&0&0&0&0&1\\
0&1&0&0&1&0\\
0&0&1&1&0&0\end{pmatrix} . 
\] 
This implies that the regular hexagon can be described as the projection of a higher dimensional polytope with 5 facets; see Figure~\ref{hexa} for an illustration.  
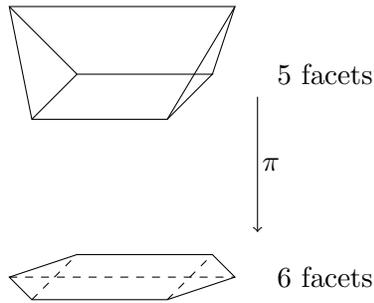
\begin{figure}
\begin{center}
\begin{tikzpicture}[scale=0.6]
	\draw (0,0) -- (0.5,-0.5) -- (3.5,-0.5) -- (5,0) -- (4.5,0.5) -- (1.5,0.5) -- cycle;
	\draw[dashed] (0.5,-0.5) -- (1.5,0.5);
	\draw[dashed] (3.5,-0.5) -- (4.5,0.5);
	\draw[dashed] (0,0) -- (5,0);
	\node at (7,0) {6 facets};
	\draw (3.5,3.5) -- (0.5,3.5) -- (0,6) -- (5,6) -- cycle;	
	\draw (0.5,3.5) -- (1.5,4.5) -- (0,6);
	\draw (1.5,4.5) -- (4.5,4.5) -- (3.5,3.5);
	\draw (4.5,4.5) -- (5,6);
	\node at (7,4.5) {5 facets};
	\draw[->] (5.5,4) -- (5.5,1);
	\node at (5.8,2.5) {$\pi$};
\end{tikzpicture}
\caption{Minimum-size extension of the regular hexagon.}
\label{hexa}
\end{center}
\end{figure} 
In this paper, we focus on the extension complexity of regular $n$-gons, and in particular on a new upper bound.

\paragraph{Extension complexity of regular $n$-gons} In the remainder of this paper, we denote $S_n$ the slack matrix of the regular $n$-gon (more precisely, any slack matrix; see Section~\ref{slackma} for a construction), hence $\rank_+(S_n)$ equals the extension complexity of the regular $n$-gon; see above. In the following, 
we describe several bounds for the nonnegative rank, focusing on the slack matrices of regular $n$-gons.

\emph{Lower bounds.} There exist several approaches to derive lower bounds for the nonnegative rank, which we classify in three classes: 
\begin{itemize}
\item \emph{Geometric}. 
Using a counting argument and the facts that 
(i) any face of a polytope is the projection of a face of its extension, and 
(ii) any face is an intersection of facets, 
it can be shown that $\rank_+(S_n) \geq \left\lceil \log_2(2n+2) \right\rceil$~\cite{G09}. 
Based on a refined geometric counting argument, Gillis and Glineur~\cite{GG10b} described a stronger lower bound for the slack matrix of polygons\footnote{They actually derived this bound for linear Euclidean distance matrices, but it also applies to the slack matrix of polygons.}: the nonnegative rank $r_+ = \rank_+(S_n)$ of $S_n$ must satisfy 
\[
n \leq \max_{3 \leq d \leq r_+-1} \; \min_{i=0,1} \; \faces(r_+,d-1,d-3+i), 
\] 
where the quantity $\faces(v,d,k)$ is the maximal number of $k$-faces of a polytope with $v$ vertices in dimension $d$, attained by cyclic polytopes~\cite{McM70}; see also~\cite[p.257, Corollary 8.28]{Z95}. We have  
\[
\faces(v,d,k-1) = \sum_{i = 0}^{\frac{d}{2}}{}^* \bigg(  
\binom{d-i}{k-i} + \binom{i}{k-d+i} \bigg) \binom{v-d-1+i}{i},  
\]
where $\sum{}^*$ denotes a sum where only half of the last term is taken for $i = \frac{d}{2}$ if $d$ is even, and the whole last term is taken for $i = \lfloor \frac{d}{2} \rfloor = \frac{d-1}{2}$ if $d$ is odd. 
This bound can be generalized to any nonnegative matrix~\cite{GG10b}, 
but it becomes difficult to compute for non-slack matrices as it requires another quantity that is in general NP-hard to compute (namely, the restricted nonnegative rank, which is always equal to $n$ for the slack matrix of a polytope with $n$ vertices).

\item \emph{Combinatorial}. These bounds are based on the sparsity pattern of the input matrix. 
The most well-known one is the rectangle covering bound (RCB) that counts the minimum number of rectangles necessary to cover all positive entries of the matrix, a rectangle being a subset of rows and columns for which the corresponding submatrix contains only positive entries; 
see~\cite{FK11} and the references therein.  Note that the RCB is equal to the boolean rank; see, e.g.,~\cite{CGP81}.  
A closely related bound is the refined rectangle covering bound (RRCB) by Oelze, Vandaele, Weltge \cite{OVW14}: in addition to covering every positive entry by a rectangle, the RRCB requires that every 2-by-2 nonsingular submatrix is touched by at least two rectangles (note that the same rectangle can be used twice). For example, the RCB for the matrix
\[
S_9 = \begin{pmatrix}
			1 & 2 & 0 & 3\\
			4 & 5 & 6 & 0\\
			7 & 8 & 9 & 0\\
		\end{pmatrix}
		\] 
is equal to two while the RRCB is equal to three. In fact, there are only three maximal rectangles (that is, rectangles not contained in any larger rectangle): 
 \[
\begin{pmatrix}
			1 & 1 & 0 & 0\\
			1 & 1 & 0 & 0\\
			1 & 1 & 0 & 0\\
		\end{pmatrix}, 
		\begin{pmatrix}
			1 & 1 & 0 & 1\\
			0 & 0 & 0 & 0\\
			0 & 0 & 0 & 0\\
		\end{pmatrix}, \text{ and } 
		\begin{pmatrix}
			0 & 0 & 0 & 0\\
			1 & 1 & 1 & 0\\
			1 & 1 & 1 & 0\\
		\end{pmatrix}, 
		\] 
		and only two of them are required to cover all positive entries  (the last two, which is the unique solution) 
		while three are necessary to touching twice all rank-two positive submatrices (which is tight since this is a 3-by-4 matrix), e.g., the block $\begin{pmatrix}
			4 & 5 \\
			7 & 8\\
		\end{pmatrix}$ touched only once with the RCB solution.   
		
		Although these bounds can be rather strong in some cases, they are computationally very expensive, and only work well for matrices with `well located' zero entries.  For the slack matrices of the regular $n$-gons, we could compute them up to $n=13$ (for larger $n$, it would take several weeks of computation with our current formulation). 

\item \emph{Convex Relaxations}. Fawzi and Parrilo developed two lower bounds for the nonnegative rank based on a sum-of-squares approximation of the copositive cone \cite{FP12, FP14}. 
These bounds are very general as they can be computed for any nonnegative matrix; however they are typically weaker than the aforementioned lower bounds, in particular for slack matrices. 

\end{itemize} 
These bounds are compared for the regular $n$-gons on 
\begin{figure}
	\vspace{-2cm}	\hspace{-2.35cm}	\includegraphics[width=21cm]{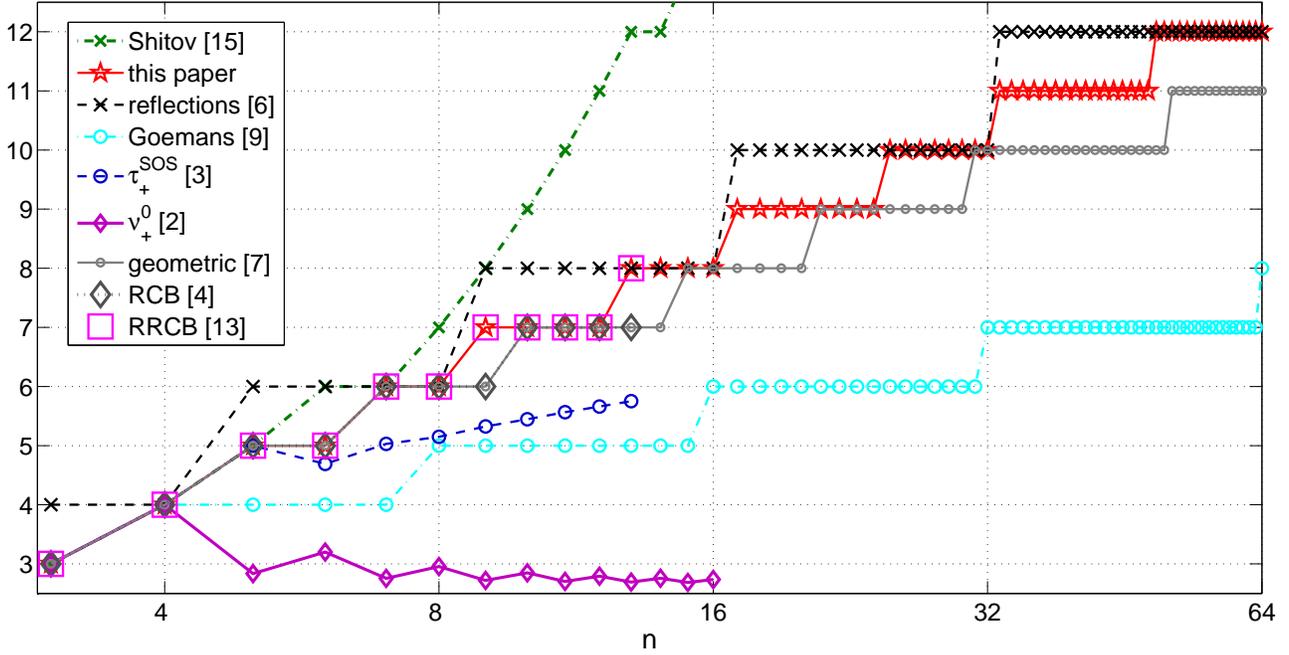}
	\vspace{-1cm}
\caption{Comparison of lower and upper bounds for the nonnegative rank of the slack matrices of regular $n$-gons, that is, $\rank_+(S_n)$. 
(Note that some bounds cannot be computed for all $n$ because of their high computational cost.) }
\label{compabounds}
\end{figure}
 Figure~\ref{compabounds}. 
We observe that the best lower bounds are the geometric bound from~\cite{GG10b} and the rectangle covering bounds~\cite{FK11, OVW14} that coincide except for $n = 9, 13$ for which only the RRCB is tight (as it matches the best upper bound; see below). \\


\emph{Upper bounds.} 
Ben-Tal and Nemirovski~\cite{BTN01} gave an extension of the regular $n$-gons when $n$ is a power of two ($n = 2^k$ for some $k$) with $2 \log_2(n) + 4$ facets. They used this construction to approximate the circle with regular $n$-gons which allowed them to approximate second-order cone programs with linear programs. 
This construction was slightly reduced to size $2 \log_2(n)$ in~\cite{G00} (again, only for $n=2^k$). 
Kaibel and Pashkovich~\cite{KP11, KP12} proposed a general construction for arbitrary $n$ of size $2 \left\lceil \log_2(n) \right\rceil +2$. 
Fiorini, Rothvoss and Tiwary~\cite{FRT12} improved the bound to $2 \left\lceil \log_2(n) \right\rceil$, which is, to the best of our knowledge, the best known upper bound for regular $n$-gons. These last bounds are based on a geometric argument using successive reflections to construct the regular $n$-gon. Note that Shitov~\cite{Shit14} proved an upper bound of $\left\lceil  \frac{6n}{7}  \right\rceil$ for the nonnegative rank of any $n$-by-$n$ rank-three nonnegative matrix, hence is applicable to the slack matrix of polygons. \\

As shown on Figure~\ref{compabounds}, prior to our new upper bound, the exact value of $\rank_+(S_n)$ is not known for most values of $n$ larger than 9 as the best lower and upper bounds do not coincide. 
Therefore, the exact value of the extension complexity of many regular $n$-gons is still unknown. 

Table~\ref{boundNR} also gives the best upper and lower bounds for $n$ up to 20. 
\begin{table}[ht!]
\begin{center}
\begin{tabular}{|c|cccccccccccccccc|}
\hline
 $n$                  &   6 & 7 & 8 & 9 & 10  &11 & 12 & 13 & 14 & 15 & 16 & 17 & 18 & 19 & 20 & 21 \\ \hline \hline 
 RRCB  \cite{OVW14}   &   \textbf 5 & \textbf  6 & \textbf  6 & \textbf 7 & \textbf  7   &\textbf 7  & \textbf  7  & \textbf  8  & ?  & ? & ? & ? & ? & ? & ? & ? \\
 geometric \cite{GG10b} &  \textbf 5 & \textbf  6 & \textbf  6 & 6 & \textbf 7  &\textbf 7   & \textbf   7   & 7 & 7 &\textbf  8 & \textbf  8 & 8 & 8 & 8 & 8 & \textbf  9 \\ 
 \hline
 Equation~\eqref{ubbest} &  \textbf 5 & \textbf  6 & \textbf  6 & \textbf  7 & \textbf  7   & \textbf  7  & \textbf  7  & \textbf  8  &  8  & \textbf  8  & \textbf  8  & 9  & 9  & 9  & 9 & \textbf  9 \\
\hline
\end{tabular}
\caption{Comparison of two lower bounds (first two rows) and the upper bound from Equation~\eqref{ubbest} for the nonnegative rank of regular $n$-gons. Bold indicates the tight bounds, that is, bounds that coincide with the nonnegative rank.}  
\label{boundNR}
\end{center}
\end{table}

\paragraph{Contribution of the Paper} 

In this paper, our contribution is mainly towfold. First, in Section~\ref{srcb}, we derive an improved lower bound for the rectangle covering number $r$ of the slack matrix of regular $n$-gons. We show that the following relation holds 
\[ 
n \leq \frac{r - \lfloor r/2 \rfloor}{r - 1} \binom{r}{ \lfloor r/2 \rfloor} , 
\]  
which improves over the best known previous relation given by $n \leq  \binom{r}{ \lfloor r/2 \rfloor}$~\cite{CGP81}. 
Although this new lower bound does not improve the best known lower bounds for the nonnegative rank of the slack matrices of regular $n$-gons (namely, the RRCB and the geometric bound; see previous paragraph), it is applicable to a broader class of matrices, namely those which have the same sparsity pattern as the slack matrices of $n$-gons. Moreover, it turns out to be a tight bound for the rectangle covering number, a.k.a.\@ the boolean rank, for some $n$ (comparing it with the upper bound from~\cite{BHJL86}). 

Second, we slightly improve the upper bound of Fiorini, Rothvoss and Tiwary~\cite{FRT12}. 
Although our approach is equivalent to that of Fiorini et al., both being recursive, our proof is rather different, being purely algebraic as opposed to their geometric approach.  Moreover, we are able to reduce the upper bound by one when $2^{k-1}  <  n  \leq 2^{k-1}+2^{k-2}$ for some $k$: this is possible by stopping the recursion earlier at a better base case (note that it would be possible to modify the proof of Fiorini et al.\@ to achieve the same bound).  
 We show that for all $n \geq 2$, 
\begin{equation} \label{ubbest}
    \rank_+(S_n) \leq  
		 \left\{ \begin{array}{ccccc} 
		2 \lceil \log_2(n) \rceil  - 1 = 2k-1  & \quad \text{for } & 2^{k-1}       & < \quad n  \quad \leq & 2^{k-1}+2^{k-2} , \\ 
		2 \lceil \log_2(n) \rceil  = 2 k   & \quad \text{for } & 2^{k-1}+2^{k-2} & <  \quad n   \quad \leq &  2^{k} . 
  \end{array} \right. 
\end{equation}
Although the improvement is relatively minor, our numerical experiments strongly suggest that this bound is tight; see the discussion at the end of Section~\ref{facto}. 
Moreover, our bound allows us to close the gap for several $n$-gons as it matches the best known lower bound, 
for $9 \leq n \leq 12$ our bound implies that $\rank_+(S_n) = 7$ and, for $21 \leq n \leq 24$, that $\rank_+(S_n) = 9$; 
see Figure~\ref{compabounds}. 
(Note that, for $n = 13$, the RRCB was, to the best of our knowledge, never computed prior to this work hence it is also the first time $\rank_+(S_n) = 8$ is claimed for $n = 13$.) \\

The paper is organized as follows. In Section~\ref{slackma}, we briefly describe the construction of the slack matrices of regular $n$-gons. In Section~\ref{srcb}, we describe our new improved lower bound for the rectangle covering of these matrices, and, in Section~\ref{facto}, we describe  our construction that proves the aforementioned upper bound. 
Then we discuss some directions for further research and conclude in Section~\ref{conclu}.

\section{The Slack Matrices of Regular $n$-gons} \label{slackma}

Let us construct the slack matrices of regular $n$-gons. 
Without loss of generality (w.l.o.g.), we use regular $n$-gons centered at the origin with their vertices located on the unit circle of radius equal to one; see Figure~\ref{smregn} for an illustration with the pentagon.  
\begin{figure}[ht!]
		\begin{center}
			\includegraphics[width=5cm]{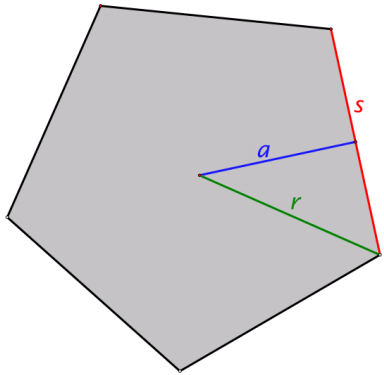}
		\end{center}
		\caption{Illustration for the construction of the slack matrices of regular $n$-gons. In this paper, we assume w.l.o.g.\@ that $r = 1$.}
	 \label{smregn} 
\end{figure}
 The length $s$ of the facets of the regular $n$-gon is given by $s=2\sin\left(\frac{\pi}{n}\right)$. 
The slack between a facet and the $k$th vertex 
(the $0$th and $(n$-$1)$th being on the considered facet, and counting along the circle in any direction) 
is equal to: 
\begin{equation} \label{cnk}
c_{{k}} = \cos\left(\frac{\pi}{n}\right) - \cos\left( (2{k}+1) \frac{\pi}{n}\right)  .  
\end{equation}
By symmetry, 
(i) our slack matrices of regular $n$-gons   are circulant matrices for which the vector $c$ is translated one element to the right on each row, and 
(ii) the vector $c$ satisfies $c_k = c_{n-1-k}$ for all $k$. 
For example, for $n=9$, we have 
		\begin{equation} \label{s9}
		S_9 = \begin{pmatrix}
			0   & c_1 & c_2 & c_3 & c_4 & c_3 & c_2 & c_1 & 0   \\
			0   & 0   & c_1 & c_2 & c_3 & c_4 & c_3 & c_2 & c_1 \\
			c_1 & 0   & 0   & c_1 & c_2 & c_3 & c_4 & c_3 & c_2 \\
			c_2 & c_1 & 0   & 0   & c_1 & c_2 & c_3 & c_4 & c_3 \\
			c_3 & c_2 & c_1 & 0   & 0   & c_1 & c_2 & c_3 & c_4 \\
			c_4 & c_3 & c_2 & c_1 & 0   & 0   & c_1 & c_2 & c_3 \\
			c_3 & c_4 & c_3 & c_2 & c_1 & 0   & 0   & c_1 & c_2 \\
			c_2 & c_3 & c_4 & c_3 & c_2 & c_1 & 0   & 0   & c_1 \\
			c_1 & c_2 & c_3 & c_4 & c_3 & c_2 & c_1 & 0   & 0   \\
		\end{pmatrix}. 
		\end{equation}
Note that, to the best of our knowledge,  the best known lower (resp.\@ upper) bound for $\rank_+(S_9)$ is 7 (resp.\@ 8). In this paper, we will improve the upper bound to 7 hence proving that $\rank_+(S_9) = 7$; see Figure~\ref{compabounds}.

\section{Lower bound for the boolean rank of $S_n$} \label{srcb}

In this section, we improve the lower bound on the boolean rank (or, equivalently, the rectangle covering number) for regular $n$-gons. On the way, we derive several new interesting results that could be used to derive other bounds. 

Let $U,V \geq 0$ be an exact nonnegative factorization of $M = UV$ of size $r$. 
In this section, we will use the following notation. Let us define the following subsets of $\{1,2,\dots r\}$, representing the supports of the rows of $U$ and columns of $V$:  
\[
s_i = \{ k \ | \ U_{ik} \neq 0 \} \; 1 \leq i \leq m 
\quad 
\text{ and } 
\quad 
t_j = \{ k \ | \ V_{kj} \neq 0 \} \; 1 \leq j \leq n . 
\]
Since $M_{ij} = U(i,:)V(:,j)$, $U \geq 0$ and $V \geq 0$, we have 
\begin{equation} \label{Mijst}
M_{ij} = 0 \quad \iff \quad s_j \cap t_j = \emptyset. 
\end{equation}
If $s_i \subseteq s_l$ for some $i,l$, \eqref{Mijst} implies that the sparsity pattern of the $i$th row of $M$ is contained in the sparsity pattern of the $l$th row of $M$ (and similarly for the columns).  
Therefore, if $M$ contains $p$ rows whose sparsity patterns are not contained in one another, 
there are $p$ subsets from $s_i$ $(1 \leq i \leq n)$ that form a Sperner family of size $p$, also know as an antichain of size $p$, which is a family of $p$ sets that are not contained in one another~\cite{S28}. 
By symmetry, the same holds for the columns.

\subsection{Sperner theorem and rectangle covering} 

Sperner theorems bounds the size of an antichain over $r$ elements. Let us recall this result and a proof that will be useful later. 


\begin{theorem} \label{spern}
Let $\mathcal{S} = \{ s_1, s_2, \dots, s_n \}$ be a set of $n$ subsets of $\{1,2,\dots,r\}$. 
Let also $\mathcal{S}$ be an antichain, that is, no subset in $\mathcal{S}$ is contained in another subset in $\mathcal{S}$. Then, 
\begin{equation} \label{spernform}
n \leq \binom{r}{ \lfloor r/2 \rfloor} , 
\end{equation}
and the bound is tight (take all subsets of size $\lfloor r/2 \rfloor$). 
\end{theorem} 
\begin{proof}(\cite{Lu66}) This proof is based on a counting argument using the fact that there are $r!$ permutations of $\{1,2,\dots,r\}$. 
Given $s_i \in \mathcal{S}$ with $k$ elements, there are $k! (r-k)!$ permutations of $\{1,2,\dots,r\}$ whose first $k$ elements are in $s_i$. 
 Because the $s_i$'s are not contained in one another, the permutations generated for two different subsets $s_i$ and $s_j$ cannot coincide (otherwise this would imply that $s_i \subset s_j$ or $s_j \subset s_i$).  
Let us also denote $a_k$ the number of sets with $k$ elements contained in $\mathcal{S}$, that is, 
$a_k = |\{s \in S \ | \ |s| = k\}|$, 
hence $n = \sum_{k=0}^r a_k$. We have 
\[
\sum_{k=0}^r a_k k! (r-k)! \quad \leq \quad  r! \; . 
\]
Therefore, 
\[
\frac{n}{ \binom{r}{ \lfloor r/2 \rfloor}} = 
\sum_{k=0}^r \frac{a_k}{ \binom{r}{ \lfloor r/2 \rfloor} }  
\quad \leq \quad 
\sum_{k=0}^r \frac{a_k}{ \binom{r}{k} }  = \sum_{k=0}^r a_k \frac{k! (r-k)!}{r!} \quad \leq \quad  1. 
\]
since $\binom{r}{ \lfloor r/2 \rfloor} \geq \binom{r}{k}$  for all $k$. This completes the proof. 
\end{proof}

The above result was used to prove that the rectangle covering of the $n$-by-$n$ Euclidean distance matrices (with zeros only the diagonal)  is the minimum $r$ such that  $n \leq \binom{r}{ \lfloor r/2 \rfloor}$; see~\cite{BL09} and the references therein. 
This result can actually be generalized for any nonnegative matrix. 
\begin{corollary}[\cite{CGP81}] \label{cor1}
Let $M$ be a matrix having $p$ rows or $p$ columns whose sparsity patterns are not contained in one another. Then, 
\[ 
\text{rc}(M) \geq \min \left\{ r \Big| \binom{r}{ \lfloor r/2 \rfloor} \geq p \right\}. 
\] 
\end{corollary}
\begin{proof} Let $M$ have $p$ rows with different sparsity patterns. As explained in the introduction of this section, this implies that  there are $p$ subsets of $\{1,2,\dots,r\}$ corresponding to the sparsity patterns of $p$ rows of $U$ that are not contained in one another. Theorem~\ref{spern} allows to conclude.   
\end{proof} 

In particular, this result can be applied to the slack matrix of any polytope. In fact, the slack of two different vertices cannot be contained in one another, otherwise it would mean that a vertex is the intersection of a subset of the facets intersecting at another vertex. 
 The same holds for two different facets by polar duality or a similar argument.
\begin{corollary} \label{cor2}
Let $M$ be the slack matrix of a polytope with $f$ facets and $v$ vertices. Then, 
\[ 
\text{rc}(M) \geq \min \left\{ r \Big| \binom{r}{ \lfloor r/2 \rfloor} \geq \max(f,v) \right\}. 
\] 
\end{corollary}

Note that, the results from Corollaries~\ref{cor1} and~\ref{cor2} were already known prior to this work; see, e.g., \cite[Cor.~4.13]{GPT13} for a more general result. 

In the next section, we apply the same ideas to improve the lower bound for the rectangle covering number of the slack matrices of $n$-gons.

\subsection{Improvement for $n$-gons} 

Let $M$ be the slack matrix of a $n$-gons such that $M_{ij} = 0$ if and only if $i=j$ or $i = (j+1) \text{mod} n$ for $1 \leq i,j \leq n$; see Section~\ref{slackma}. To simplify the heavy notation $\text{mod} n$, we will assume throughout this section that $i = 1 \equiv n+1$ when $i$ represents an index.  
As before, let $UV = M$ be a nonnegative factorization of size $r$ of $M$, let $s_i$ denote the support of the $i$th row of $U$ ($1 \leq i \leq n$) and $t_j$ the support of the $j$th column of $V$ ($1 \leq j \leq n$).  
We have $M_{ij} = 0$ if and only if $i=j$ or $i = j+1$, and 
\[
M_{ij} = 0 \quad \iff \quad s_i \cap t_j = \emptyset. 
\] 
Let us try to characterize the size of the sets $\mathcal{S} = \{ s_1, s_2, \dots, s_n \}$ 
and $\mathcal{T} = \{ \bar{t_1}, \bar{t_2}, \dots, \bar{t_n} \}$ that satisfy the above property, where $\bar{t_j}$ denotes the complement of $t_j$. 

First, we can assume without loss of generality that 
$t_i = \overline{s_i \cup s_{i+1}}$. In fact, $t_i = \overline{s_i \cup s_{i+1}}$ is the largest possible set that does not intersect $s_i \cup s_{i+1}$ while having the most intersections with all other sets in $\mathcal{S}$ (which is the best possible situation since $M_{ij} > 0$ for $i\neq j, j+1$). 

For the same reason as in Corollary~\ref{cor1}, since the rows and columns of $M$ have different sparsity patterns, we have that 
\begin{itemize} 

\item[(C1)] $\mathcal{S} = \{ s_1, s_2, \dots, s_n \}$ is an antichain. 

\item[(C2)] $\mathcal{T} = \{ s_1 \cup s_2, s_2 \cup s_3, \dots, s_{n-1} \cup s_n, s_n \cup s_1 \}$ is an antichain, since taking the complement of all the sets in an antichain gives another antichain of the same size.  

\item[(C3)] Every set $s_i \subseteq \{1,2,\dots,r\}$ contains at least one element not in the sets $\bar{t}_j = s_j \cup s_{j+1}$ for $j,j+1 \neq i$, since $M_{ij} > 0$ for $i\neq j, j+1$.   
\end{itemize}

\begin{theorem} \label{mainlow}
Let $\mathcal{S}$ and $\mathcal{T}$ satisfy (C1-C3) and $r \geq 2$. Then 
\begin{equation} \nonumber 
n \leq \frac{r - \lfloor r/2 \rfloor}{r - 1} \binom{r}{ \lfloor r/2 \rfloor} . 
\end{equation}
\end{theorem}

\begin{proof} 
Let us denote 
$k_i$ the number of elements in $s_i$, 
$z_i$ the number of additional elements in $\bar{t}_i$ compared to $s_i$ (that is, $|\bar{t}_i| = k_i + z_i$) and 
$z_i'$ the number of additional elements in $\bar{t}_{i-1}$ compared to $s_i$ (that is, $|\bar{t}_{i-1}| = k_i + z'_i$). 
Following the same argument as in Theorem~\ref{spern}, we have that  
the number of permutations  
with the elements of $s_i$ in the first positions is given by $k_i! (r-k_i)!$, 
of $\bar{t}_i$ by $(k_i+z_i)! (r-k_i-z_i)!$, and 
of $\bar{t}_{i-1}$ by $(k_i+z_i')! (r-k_i-z_i')!$. However, between $s_i$ and $\bar{t}_i$, 
there are $k_i! z_i! (r-k_i-z_i)!$ common permutations (and similarly between $s_i$ and $\bar{t}_{i-1}$). 
Note that these are the only possible repetitions because of (C3). 
Note also that $|\bar{t}_{i}| = k_{i} + z_{i} = k_{i+1}+z_{i+1}'$ hence the number of permutations corresponding to $\bar{t}_i$ are also equal to $1/2(k_i! z_i! (r-k_i-z_i)! + (k_{i+1}! z_{i+1}! (r-k_{i+1}-z_{i+1}))$. 
Counting all permutations corresponding to $s_i$ and $\bar{t}_i$ for $1 \leq i \leq n$ and accounting for the repetitions, we get 
\[
\sum_{i=1}^n 
k_i! (r-k_i)! 
+ \frac{1}{2} (k_i+z_i)! (r-k_i-z_i)! 
+ \frac{1}{2} (k_i+z_i')! (r-k_i-z_i')! 
- k_i! z_i! (r-k_i-z_i)! - k_i! z_i'! (r-k_i-z_i')!
\leq r! . 
\]
Let us lower bound the left hand side of the above inequality. To do so, we minimize over each term of the sum independently. 
Noting that $z_i$ and $z_i'$ have exactly the same role, we can assume without loss of generality that $z_i = z'_i$ at a minimum.  
Removing the index $i$ for simplicity, we therefore have to evaluate 
\[ 
\min_{k\geq 1, z \geq 1, k+z \leq r} k! (r-k)! 
+ (k+z)! (r-k-z)! - 2 k! z! (r-k-z)! . 
\] 
In Appendix~\ref{appA}, we show that $k^* = \lfloor r/2 \rfloor$ and $z^* = 1$ is an optimal solution. 
Therefore, dividing the inequality above by $r!$ and using our lower bound for each term 
(replacing the $k_i$'s with $\lfloor r/2 \rfloor$ and the $z
_i$'s with 1), we obtain 
\[
n \left( \binom{r}{ \lfloor r/2 \rfloor}^{-1} 
+ 
\underbrace{\binom{r}{ \lfloor r/2 \rfloor}^{-1} \frac{\lfloor r/2 \rfloor + 1}{r - \lfloor r/2 \rfloor}}_{\binom{r}{ \lfloor r/2 \rfloor + 1}^{-1}} \left( 1 - 2 \frac{1}{\lfloor r/2 \rfloor + 1}\right)  \right) \leq 1 . 
\]
from which we get, after simplifications, $n \leq \frac{r - \lfloor r/2 \rfloor}{r - 1} \binom{r}{ \lfloor r/2 \rfloor}$.
\end{proof}

\begin{corollary} \label{maincor}
Let $r$ be the rectangle covering number of the slack matrix of any $n$-gon for $n \geq 2$, then 
\[ 
n \; \leq \;  \frac{r - \lfloor r/2 \rfloor}{r - 1} \binom{r}{ \lfloor r/2 \rfloor}  . 
\] 
\end{corollary}

Note that the term $\frac{r - 1}{r - \lfloor r/2 \rfloor}$ goes to 1/2 when $r$ grows, and we cannot hope to obtain a better bound using our counting argument. In fact, this is the case when there would be no repetitions between the permutations 
generated from the sets in $\mathcal{S}$ and $\mathcal{T}$; see the proof of Theorem~\ref{mainlow}. 

The bound from the corollary above also applies to the so-called boolean rank, which is the same as the rectangle coreving number. 
Comparing our bound with the upper bounds computed in~\cite[p.145]{BHJL86} for small $n$, 
our bound is tight for $n = 2-6, 8-9, 13-21, 24-32$ ($a-b$ means from $a$ to $b$, that is, $a,a+1,\dots,b$), which was not the case of the previous bound~\eqref{spernform} which is tight only for $n = 2-4$.

\section{Explicit nonnegative factorization of slack matrices $S_n$ of regular $n$-gons} \label{facto}  

In this section, we construct a nonnegative factorization of $S_n$ in a recursive way. 
The idea is the following. 
At the first step, 
a rank-two modification of $S_n$ is performed so that 
the pattern of zero entries of the constructed matrix therefore looks like a cross (see below for an example on $S_9$). 
This subdivides the matrix into four blocks with a lot of symmetry that implies that the nonnegative rank of one block equals the nonnegative rank of the full matrix. Then, the same scheme is applied to that subblock until the number of columns of the obtained block $B$ is smaller than four, which we factorize with a trivial decomposition $B = B I$ ($I$ being the identity matrix of appropriate dimension). 

Before we rigorously prove that our construction works for any $n$-gon, let us illustrate the idea on the slack matrix of the regular 9-gon form~\eqref{s9}. 
Observe that the entries of the slack matrix on the main diagonal and the diagonal below it are equal to zero. 
The first step of our construction will make a rank-two correction of the slack matrix so that the same pattern appears: 
we remove a matrix from the 4-by-4 lower left block of $S_9$ 
    $$\begin{pmatrix}
			 {c_4} &  {c_3} &  {c_2} &  {\underline{c_1}} \\
			 {c_3} &  {c_4} &  {\underline{c_3}} &  {\underline{c_2}} \\
			 {c_2} &  {\underline{c_3}} &  {\underline{c_4}} &  {c_3} \\
			 {\underline{c_1}} &  {\underline{c_2}} &  {c_3} &  {c_4} \\
	  \end{pmatrix} -
	  \begin{pmatrix}
			 {c_4-c_3} &  {c_3-c_2} &  {c_2-c_1} &  {\underline{c_1}} \\
			 {c_3-c_2} &  {c_4-c_1} &  {\underline{c_3}} &  {\underline{c_2}} \\
			 {c_2-c_1} &  {\underline{c_3}} &  {\underline{c_4}} &  {c_3-c_1} \\
			 {\underline{c_1}} &  {\underline{c_2}} &  {c_3-c_1} &  {c_4-c_2} \\
	  \end{pmatrix} = 
	  \begin{pmatrix}
			 {c_3} &  {c_2} &  {c_1} &  {\underline{0}} \\
			 {c_2} &  {c_1} &  {\underline{0}} &  {\underline{0}} \\
			 {c_1} &  {\underline{0}} &  {\underline{0}} &  {c_1} \\
			 {\underline{0}} &  {\underline{0}} &  {c_1} &  {c_2} \\
	  \end{pmatrix}, 
		$$
and another matrix from the positive 4-by-4 block of $S_9$ at the upper right (rows 2 to 5, last 4 columns) 
		$$\begin{pmatrix}
			 {c_4} &  {c_3} &  {\underline{c_2}} &  {\underline{c_1}} \\
			 {c_3} &  {\underline{c_4}} &  {\underline{c_3}} &  {c_2} \\
			 {\underline{c_2}} &  {\underline{c_3}} &  {c_4} &  {c_3} \\
			 {\underline{c_1}} &  {c_2} &  {c_3} &  {c_4} \\
	  \end{pmatrix} -
	  \begin{pmatrix}
			 {c_4-c_2} &  {c_3-c_1} &  {\underline{c_2}} &  {\underline{c_1}} \\
			 {c_3-c_1} &  {\underline{c_4}} &  {\underline{c_3}} &  {c_2-c_1} \\
			 {\underline{c_2}} &  {\underline{c_3}} &  {c_4-c_1} &  {c_3-c_2} \\
			 {\underline{c_1}} &  {c_2-c_1} &  {c_3-c_2} &  {c_4-c_3} \\
	  \end{pmatrix} = 
	  \begin{pmatrix}
			 {c_2} &  {c_1} &  {\underline{0}} &  {\underline{0}} \\
			 {c_1} &  {\underline{0}} &  {\underline{0}} &  {c_1} \\
			 {\underline{0}} &  {\underline{0}} &  {c_1} &  {c_2} \\
			 {\underline{0}} &  {c_1} &  {c_2} &  {c_3} \\
	  \end{pmatrix}.
		$$ 
		Clearly, the removed matrices are nonnegative since $0 \leq c_{k-1} \leq c_{k}$ for all $0 \leq k \leq \lfloor \frac{n}{2} \rfloor$. Moreover, we show in the next lemma that they have rank one. 
		\begin{lemma} \label{lemma1}
		The (infinite) matrix
\[ 
\Bigl[ c_{\alpha-i+j} - c_{\beta-i-j} \Bigr]_{i \in \mathbb{Z}, j \in \mathbb{Z}} 
\]  
has rank one for any fixed $\alpha \in \mathbb{Z}$, $\beta \in \mathbb{Z}$ and $n \in \mathbb{N}_{>0}$. 
		\end{lemma}
		\begin{proof}
We have that $c_k = \cos(\frac{\pi}{n}) -  \cos( (2k+1) \frac{\pi}{n}) = 2 \sin(k \frac{\pi}{n})  \sin((k+1) \frac{\pi}{n})$. 
Choosing any $2 \times 2$ minor with rows $i \in \{0, x\}$ and columns $j \in \{0, y\}$ (w.l.o.g.), one can check, using algebra with a few trigonometric identities, that the determinant of 
\[ 
\begin{pmatrix} 
c_{\alpha} - c_{\beta} & c_{\alpha+y} - c_{\beta-y}  \\ 
c_{\alpha-x} - c_{\beta-x}  & c_{\alpha-x+y} - c_{\beta-x-y} 
\end{pmatrix} 
\] 
is equal to zero for any $x$, $y$, and any $n$. 
\end{proof}

After these two nonnegative rank-one factors are removed, we obtain 
    $$
		S_9 - 
		\begin{pmatrix}
		0 \\ 0 \\ 0  \\ 0 \\ 0 \\c_1   \\  c_2  \\ c_3-c_1 \\  c_4 - c_2 \\ 
		\end{pmatrix}
		\begin{pmatrix}
		\frac{c_4-c_3}{c_1} \\ \frac{c_3-c_2}{c_1} \\ \frac{c_2 - c_1}{c_1}  \\ 1 \\ 0 \\ 0 \\ 0  \\ 0 \\ 0 \\ 
		\end{pmatrix}^T 
		- 
		\begin{pmatrix}
		  0 \\ c_4 - c_2 \\ c_3 - c_1  \\ c_2  \\  c_1 \\ 0 \\ 0 \\ 0  \\ 0 \\ 
		\end{pmatrix}
		\begin{pmatrix}
		 0 \\ 0 \\ 0  \\ 0 \\ 0 \\ 1 \\ \frac{c_2 - c_1}{c_1}  \\ \frac{c_3 - c_2}{c_1}   \\ \frac{c_4 - c_3}{c_1} \\  
		\end{pmatrix}^T
		= 
		\left( 
		\begin{array}{ccccc|cccc}
			0   & c_1 & c_2 & c_3 & c_4 & c_3 & c_2 & c_1 & 0   \\
			0   & 0   & c_1 & c_2 & c_3 & c_2 & c_1 & 0   & 0   \\
			c_1 & 0   & 0   & c_1 & c_2 & c_1 & 0   & 0   & c_1 \\
			c_2 & c_1 & 0   & 0   & c_1 & 0   & 0   & c_1 & c_2 \\
			c_3 & c_2 & c_1 & 0   & 0   & 0   & c_1 & c_2 & c_3 \\ \hline 
			c_3 & c_2 & c_1 & 0   & 0   & 0   & c_1 & c_2 & c_3 \\
			c_2 & c_1 & 0   & 0   & c_1 & 0   & 0   & c_1 & c_2 \\
			c_1 & 0   & 0   & c_1 & c_2 & c_1 & 0   & 0   & c_1 \\
			0   & 0   & c_1 & c_2 & c_3 & c_2 & c_1 & 0   & 0   \\
		\end{array} \right) , 
		$$
		with a pattern of zeros forming a cross.  
This matrix is highly symmetric and has a lot of redundancy: 
the last four columns (resp.\@ rows) are copies of the first four. 
	Therefore, if we had a nonnegative factorization of the 5-by-5 upper left block then we would have a nonnegative factorization of the entire matrix with the same nonnegative rank. 
	
	To construct that factorization, we apply our strategy recursively: use a rank-two correction to the upper left block to make a cross of zeros appear: 
		$$
		\begin{pmatrix}
			0   & c_1 & c_2 &  {c_3} &  {c_4} \\
			0   & 0   & c_1 &  {c_2} &  {c_3} \\
			c_1 & 0   & 0   & c_1 & c_2 \\
			 {c_2} &  {c_1} & 0   & 0   & c_1 \\
			 {c_3} &  {c_2} & c_1 & 0   & 0   \\
		\end{pmatrix} 
\quad  \rightarrow \quad 
		\left( 
		\begin{array}{ccc|cc} 
			0   & c_1 & c_2 & c_1 & 0   \\
			0   & 0   & c_1 & 0   & 0   \\
			c_1 & 0   & 0   & 0   & c_1 \\ \hline 
			c_1 & 0   & 0   & 0   & c_1 \\
			0   & 0   & c_1 & 0   & 0   \\
		\end{array} \right) 
		= 
		\left( 
		\begin{array}{ccc} 
			0   & c_1 & c_2    \\
			0   & 0   & c_1   \\
			c_1 & 0   & 0    \\ \hline 
			c_1 & 0   & 0    \\
			0   & 0   & c_1    \\
		\end{array} \right)
		\left( 
		\begin{array}{ccc|cc} 
			1   & 0 & 0 &   0  & 1   \\
			0   & 1   & 0 &  1  & 0   \\
			0 & 0   & 1   &  0   & 0 
		\end{array} \right)
		.
		$$
		Now, the upper left block has a trivial nonnegative factorization (since it is a 3-by-3 matrix of rank 3) from which we can derive a nonnegative factorization for the full matrix $S_9$: 
		\[
		\begin{pmatrix}
		  0  &  c_1  &  c_2 &        0 &    c_3 - c_1  &        0     &     0 \\
  0      &   0 &   c_1     &    0  &  c_2     &    0  &  c_4-c_2 \\
    c_1  &    0      &   0   &      0  &  c_1      &   0  &  c_3 - c_1 \\
    c_1   &   0      &   0  &  c_1       &  0      &   0 &  c_2 \\
     0   &      0   & c_1  &  c_2      &   0      &   0  &  c_1 \\
      0   &      0   & c_1  &  c_2     &    0  &  c_1 &         0 \\
    c_1    &     0     &    0  &  c_1    &     0  &  c_2    &     0 \\
    c_1     &    0     &    0  &       0  &  c_1   & c_3 - c_1     &    0 \\
         0   &      0  &  c_1    &     0 &   c_2   & c_4 - c_2     &    0 \\ 
		\end{pmatrix} 
		\begin{pmatrix}
		 1 & 0     &    0    &     0 &   1 & 0  &       0   &      0  &  1  \\ 
         0  &  1 & 0   & 1 & 0&    1 & 0  &  1 & 0 \\ 
         0    &     0  &  1 & 0    &     0   &      0  &  1 & 0   &       0 \\ 
    \frac{c_2 - c_1}{c_1}  &  1 & 0    &     0   &      0    &     0    &     0   & 1 &    \frac{c_2 - c_1}{c_1} \\ 
         0   &     0   &     0  &  1  &   \frac{c_2 - c_1}{c_1}  &  1 & 0    &     0     &    0 \\ 
    \frac{c_4-c_3}{c_1} &  \frac{c_3-c_2}{c_1} &  \frac{c_2 - c_1}{c_1}  & 1 & 0 & 0 &  0 & 0 &  0 \\
        0 & 0 & 0  & 0 & 0 & 1 & \frac{c_2 - c_1}{c_1}  & \frac{c_3 - c_2}{c_1}   &  \frac{c_4 - c_3}{c_1} \\  
		\end{pmatrix} . 
		\]
		\begin{remark} Once the first two rank-one factors have been removed from $S_9$, the 5-by-5 block could also directly be trivially factorized, and we would obtain 
			\[
		S_9 = \begin{pmatrix}
	  	0   & c_1 & c_2 & c_3 & c_4 &   0     &     0  \\
			0   & 0   & c_1 & c_2 & c_3 & 0  &  c_4-c_2 \\
			c_1 & 0   & 0   & c_1 & c_2 &  0  &  c_3 - c_1 \\
			c_2 & c_1 & 0   & 0   & c_1 &  0 &  c_2 \\
			c_3 & c_2 & c_1 & 0   & 0   &  0  &  c_1  \\
			c_3 & c_2 & c_1 & 0   & 0  &  c_1 &         0 \\
			c_2 & c_1 & 0   & 0   & c_1 &   c_2    &     0 \\
			c_1 & 0   & 0   & c_1 & c_2 &  c_3 - c_1     &    0 \\
			0   & 0   & c_1 & c_2 & c_3 &  c_4 - c_2     &    0    \\
		\end{pmatrix} 
		\begin{pmatrix}
		 1 & 0     &    0    &     0 &  0 & 0  &       0   &      0  &  1  \\ 
     0  &  1 & 0   & 0 & 0&    0 & 0  &  1 & 0 \\ 
     0    &     0  &  1 & 0    &     0   &      0  &  1 & 0   &       0 \\ 
     0    &     0  &  0 & 1    &     0   &      1  &  0 & 0   &       0 \\ 
     0    &     0  &  0 & 0    &     1   &      0  &  0 & 0   &       0 \\ 
    \frac{c_4-c_3}{c_1} &  \frac{c_3-c_2}{c_1} &  \frac{c_2 - c_1}{c_1}  & 1 & 0 & 0 &  0 & 0 &  0 \\
        0 & 0 & 0  & 0 & 0 & 1 & \frac{c_2 - c_1}{c_1}  & \frac{c_3 - c_2}{c_1}   &  \frac{c_4 - c_3}{c_1} \\  
		\end{pmatrix} . 
		\]
		\end{remark}

For $n$ even, the construction slightly changes because the symmetry in the residual with the cross pattern of zero is different.  
Let us illustrate it for $n=6$. The first rank-two correction is the same as for $n = 9$ and we obtain 
		\begin{equation} 
		S_6 = \begin{pmatrix}
			0   & c_1 & c_2 & c_2 & c_1 & 0  \\
			0   & 0   & c_1 & c_2 & c_2 & c_1  \\
			c_1 & 0   & 0   & c_1 & c_2 & c_2  \\
			c_2 & c_1 & 0   & 0   & c_1 & c_2  \\
			c_2 & c_2 & c_1 & 0   & 0   & c_1 \\
			c_1 & c_2 & c_2 & c_1 & 0   & 0   \\
		\end{pmatrix} 
		\rightarrow 
		R_6 = \begin{pmatrix}
			0   & c_1 & c_2 & c_2 & c_1 & 0  \\
			0   & 0   & c_1 & c_1 & 0 & 0  \\
			c_1 & 0   & 0   & 0 & 0 & c_1  \\
			c_2 & c_1 & 0   & 0   & c_1 & c_2  \\
			c_1 & 0 & 0 & 0   & 0   & c_1 \\
			0 & 0 & c_1 & c_1 & 0   & 0   \\
		\end{pmatrix} . 
		\end{equation}
However, the fourth row of $R_6$ is not a copy of the first three. Therefore, we need to keep it: factorizing the following submatrix 
\[
R'_6 = \begin{pmatrix}
			0   & c_1 & c_2   \\
			0   & 0   & c_1   \\
			c_1 & 0   & 0     \\
			c_2 & c_1 & 0     \\
		\end{pmatrix}
\]
allows to factor $R_6$ (last three columns and last two rows are duplicates). Since it is a $4$-by-$3$ matrix, we can factor it trivially as $R'_6 = R'_6 I_3$ and obtain a rank-5 nonnegative factorization of $S_6$. 

In summary, 
\begin{itemize}

\item At the recursion steps, the factorization of the remaining $k$-by-$l$ block ($k = l$ or $l+1$) 
is computed via a nonnegative rank-two correction and the factorization of its $\lceil k' \rceil$-by-$\lceil \frac{l}{2} \rceil$ upper left block where $k' = \lceil \frac{l}{2} \rceil +1$ when $k=l$ is even and $k' = \lceil \frac{l}{2} \rceil$ otherwise.  

\item At the last step, when $k \leq 4$, a trivial factorization is used. 
Note that there will be four `basic' cases: 
3-by-3 (e.g., for $n = 5,9$),  
4-by-3 (e.g., for $n = 6$), 
4-by-4 (e.g., $n = 4, 7$), 
and 3-by-2 (e.g., for $n = 10$). 
\end{itemize}

In the recursion steps described above, 
from a large matrix with $c$ columns, a submatrix with $\lceil \frac{c}{2} \rceil$ columns is extracted, 
and the nonnegative rank of the larger matrix is smaller than that of the submatrix plus two (because of the two nonnegative rank-one corrections). 
This leads to the following result: 
  \begin{theorem} \label{mainth} 
	Let $n \geq 2$, then the nonnegative rank of any slack matrix $S_n$ of the regular $n$-gon is bounded as follows: 
	\begin{equation} \label{ubound}
    \rank_+(S_n) \leq  
		 \left\{ \begin{array}{ccccc} 
		2 \lceil \log_2(n) \rceil  - 1 = 2k-1  & \quad \text{for } & 2^{k-1}       & < \quad n  \quad \leq & 2^{k-1}+2^{k-2} , \\ 
		2 \lceil \log_2(n) \rceil  = 2 k   & \quad \text{for } & 2^{k-1}+2^{k-2} & <  \quad n   \quad \leq &  2^{k} . 
  \end{array} \right. 	
\end{equation} 
	 \end{theorem}
	\begin{proof}
	Let us first assume that the recursion described above is correct, that is, that at each step	the number of columns $c$ is decreased to $\lceil \frac{c}{2} \rceil$ while the nonnegative rank is increased by at most 2, unless $c \leq 4$ in which case we use the trivial factorization of rank $c$.  
	To verify that~\eqref{ubound} holds, we observe that the function $\lceil \frac{c}{2} \rceil$ is nondecreasing in $c$  
	hence it suffices to verify that the upper bound holds for the critical values $2^{k}, 2^{k-1}+1, 2^{k-1}+2^{k-2}$ and $2^{k-1}+2^{k-2}+1$ for any $k$. 
	For $n = 2^{k}$, 
	we check that the recursion divides the number of column by two at each step until the number of columns is equal to four which gives $\rank_+(S_n) \leq 2 \log_2(n)$. 
	For $n = 2^{k-1}+1$, 
	the number of columns $c = 2^p + 1$ for some $p$ is reduced at each step to $\lceil c / 2 \rceil = 2^{p-1} + 1$. After $k-2$ steps, we get a 3-by-3 matrix which gives $\rank_+(S_n) \leq 2 (k-2) + 3 = 2k-1$. 
	For $n = 2^{k-1}+2^{k-2}$, 
	after $k-2$ steps, the number of columns is equal to 3 hence we obtain $\rank_+(S_n) \leq  3 + 2(k-2) = 2k-1$; the case $n = 2^{k-1}+2^{k-2}+1$ is similar to that above. \\
	
	Let us now prove the recursion. To understand the proof, we encourage the reader to also look at the (short) Matlab code in Appendix~\ref{matcode} that constructs the factorizations\footnote{Note that we have numerically checked the correctness of the construction for all $n \leq 10000$.}. 
	
	Let $B$ be the $k$-by-$l$ upper left block of the slack matrix $S_n$, where $k = l$ or $l+1$ and $1 \leq k,l \leq n$. 
	Note that, at the first step, $k=l=n$. 
	
	\paragraph{Basic step.} If $l \leq 4$, $B$ is trivially factorized, that is, $B = B I_l$ where $I_l$ is the $l$-by-$l$ identity matrix.

	\paragraph{Recursion step.} If we show that 
	\[
	\rank_+(B) \leq 2 + \rank_+(B')  , 
	\]
	where $B'$ is the $k'$-by-$\lceil l/2 \rceil$ 
	upper left block of $B$, where $k'= \lceil k/2 \rceil$ except when $k = l$ is even 
	in which case $k' = \lceil k/2 \rceil + 1 = l/2 + 1$, 
	then the proof will be complete, by recursion (since $B'$ is also a $k'$-by-$l'$ upper left block of the slack matrix $S_n$ where 
	$l' = \lceil l/2 \rceil$ and $k' = l'$ or $l'+1$). 
	
	Since $B$ is the upper left block of $S_n$, it is a circulant matrix and has the following form 
	\[
	B = 
	 \begin{pmatrix}
			 c_0    & c_{1} & \dots & c_{-1+l}  \\
			 c_{-1}& c_{0}   & \dots  & c_{-2+l}  \\
			 \vdots &  \vdots & \dots  & \vdots \\
			 c_{-k+1} & c_{-k+2} & \dots &   c_{-k+l} \\ 
		\end{pmatrix} 
		= 
	\left[ c_{-i+j} \right]_{1 \leq i \leq k, 1 \leq j \leq l} , 
	\]
	where the $c_{k}$'s are given by~\eqref{cnk}. 
	The recursion works as follows. First, we subdivide the matrix $B$ into four blocks: 
	(i) the upper left $\lceil l/2 \rceil$-by-$\lfloor l/2 \rfloor$ block, 
		(ii) the upper right $\lceil l/2 \rceil$-by-$\lceil l/2 \rceil$ block,
		(iii) the lower left $(k-\lceil l/2 \rceil)$-by-$\lfloor l/2 \rfloor$ block, and 
		(iv) the lower right $(k-\lceil l/2 \rceil)$-by-$\lceil l/2 \rceil$ block. (Note that $ k - \lceil l/2 \rceil = \lfloor l/2 \rfloor + k-l$ which will be useful later.) 	
		Then, we make a nonnegative rank-one correction to the upper right and lower left blocks so that the off-diagonal entries of $B$ and the entries below are set to zero, that is, all entries $(i,j)$ of $B$ such that $i+j=l+1$ or $i+j=l+2$ will be set to zero. (Note that the entries $(i,j)$ of $B$ such that $i=j$ or $i=j+1$ are already equal to zero.) 
	
	\emph{Upper right block.} Let $p = \lceil l/2 \rceil$ and consider the $p$-by-$p$ upper right block of $B$
	\[ 
	 U = \begin{pmatrix}
			 c_{l-p}    & c_{l-p+1} & \dots & c_{l-1}  \\
			 c_{l-p-1} & c_{l-p}   & \dots  & c_{l-2}  \\
			 \vdots &  \vdots & \dots  & \vdots \\
			c_{l-2p+1} & c_{l-2p+2} & \dots &   c_{l-p} \\ 
		\end{pmatrix} 
		= 
	\left[ c_{-i+j} \right]_{1 \leq i \leq p, l-p+1 \leq j \leq l} 
	= 
	\left[ c_{-i+h+l-p} \right]_{1 \leq i \leq p, 1 \leq h = j-l+p \leq p} ,  
		\]
		from which we remove the matrix $U - [c_{1+p-i-j}]_{1 \leq i \leq p, 1 \leq j \leq p}$ which is equal to 
		\[ 
	 \begin{pmatrix}
			 c_{l-p} - c_{p-1}   & c_{l-p+1} - c_{p-2} & \dots & c_{l-1} - c_{0}  \\
			 c_{l-p-1} - c_{p-2} & c_{l-p} - c_{p-3}   & \dots  & c_{l-2} - c_{-1}  \\
			 \vdots &  \vdots & \dots  & \vdots \\
			c_{l-2p+1} - c_{0} & c_{l-2p+2} - c_{-1} & \dots &   c_{l-p} - c_{-p+1} \\ 
		\end{pmatrix} 
	= \left[ c_{\alpha-i+j} - c_{\beta-i-j} \right]_{1 \leq i \leq p, 1 \leq j \leq p} ,  
		\]
	where $\alpha = l-p$ and $\beta = 1+p$. 	By Lemma~\ref{lemma1}, that matrix has rank-one. Moreover, it is nonnegative because 
	for all $1 \leq i,j \leq p$ 
	\[
	c_{l - \lceil l/2 \rceil - i + j} 
	= 
		c_{\lfloor l/2 \rfloor - i + j}
	\geq 
	c_{1+ \lceil l/2 \rceil -i-j}
	\]
	since $\lfloor l/2 \rfloor  + j \geq 1 + \lceil l/2 \rceil -j$ for all $j$. 
	We obtain 
	\[
	\left[ c_{-i+j+l-p} - c_{\alpha-i+j} + c_{\beta-i-j} \right]_{1 \leq i \leq p, 1 \leq j \leq p} 
	=
	 \begin{pmatrix}
			 c_{p-1}   & c_{p-2} & \dots          & c_{1}  & 0  \\
			  c_{p-2} &  c_{p-3}   & \dots & 0  & 0  \\
			 \vdots &  \vdots & \dots  & \vdots & \vdots \\
			c_{1} & 0 & \dots &  c_{p-4}  &   c_{p-3} \\ 
			  0 & 0 & \dots &  c_{p-3}  &   c_{p-2} \\ 
		\end{pmatrix} 
		= \left[ c_{p+1-i-j} \right]_{1 \leq i \leq p, 1 \leq j \leq p} . 
	\]

	\emph{Lower left block.} Let $p = \lfloor l/2 \rfloor$ and $q = p + k-l = k - \lceil l/2 \rceil$ ($=p$ if $k = l$, $= p+1$ if $k = l+1$), 
	and consider the $q$-by-$p$ lower left block of $B$
	\[ 
	 L = \begin{pmatrix}
			 c_{-k+q}    & c_{-k+q+1} & \dots & c_{-k+q+p-1} \\
			 \vdots &  \vdots & \dots  & \vdots \\
			c_{-k+2} & c_{-k+3} & \dots &   c_{-k+p+1} \\ 
			c_{-k+1} & c_{-k+2} & \dots &   c_{-k+p} \\ 
		\end{pmatrix} 
		= 
	\left[ c_{-i+j} \right]_{k-q+1 \leq i \leq k, 1 \leq j \leq q} 
	= 
	\left[ c_{-h-k+q+j} \right]_{1 \leq h = i -k + q \leq q, 1 \leq j \leq p} ,  
		\]
		from which we remove the matrix $L - [c_{1+p-i-j}]_{1 \leq i \leq q, 1 \leq j \leq p}$ which is equal to
		\[ 
	 \begin{pmatrix}
			 c_{-k+q} - c_{p-1}   & c_{-k+q+p} - c_{p-2} & \dots & c_{-k+q+p-1} - c_{0}  \\
			 c_{-k+q-1} - c_{p-2} &  c_{-k+q} - c_{p-3}   & \dots  & c_{-k+q+p-2} - c_{-1}  \\
			 \vdots &  \vdots & \dots  & \vdots \\ 
			c_{-k+2} - c_{p-q+1} & c_{-k+3} - c_{p-q} & \dots &   c_{-k+p+1} - c_{-q+2} \\ 
			c_{-k+1} - c_{p-q} & c_{-k+2} - c_{p-q-1} & \dots &   c_{-k+p} - c_{-q+1} \\ 
		\end{pmatrix} 
		= 
	\left[ c_{\alpha-i+j} - c_{\beta-i-j} \right]_{1 \leq i \leq q, 1 \leq j \leq p} ,  
		\]
	where $\alpha = -k+q = -\lfloor l/2 \rfloor$ 
	and $\beta = 1+p = \lfloor l/2 \rfloor +1$, 
	which can be checked to be nonnegative (using the fact that $c_{-k} = c_{k+1}$, we have 
	 $c_{\alpha-i+j} = c_{-\alpha+i-j+1} = c_{\lfloor l/2 \rfloor+i-j+1} \geq c_{\lfloor l/2 \rfloor+1-i-j} = c_{\beta-i-j}$), 
	and has rank-one by Lemma~\ref{lemma1}. 
	We obtain 
	\[
	\left[ c_{-i-k+q+j} - c_{\alpha-i+j} + c_{\beta-i-j} \right]_{1 \leq i \leq q, 1 \leq j \leq p} 
	=
	 \begin{pmatrix}
			 c_{p-1}   & c_{p-2} & \dots          & c_{1}  & 0  \\
			  c_{p-2} &  c_{p-3}   & \dots & 0  & 0  \\
			 \vdots &  \vdots & \dots  & \vdots & \vdots \\
			c_{p-q+1} & c_{p-q} (=0) & \dots &  c_{-q+3}  &   c_{-q+2} \\ 
			  c_{p-q} (=0) & c_{p-q-1} & \dots &  c_{-q+2}  &   c_{-q+1} \\ 
		\end{pmatrix} . 
	\]
	Note that, if $k = l$ (that is, $p = q$) then $c_{p-q-1} = 0$ otherwise $k = l+1$ and $c_{p-q+1} = 0$. \\ 

		Finally, putting all the blocks together: 
		 the untouched upper left and lower right blocks, and the corrected upper right and lower left blocks, 	
		we obtain, after a nonnegative rank-two correction of $B$, the following $l$-by-$l$ matrix 
			\[
	 \begin{pmatrix}
0   & c_1    & c_2    & \dots    & c_2     & c_1 & 0 \\
0   & 0      & c_1    & \dots    & c_1     & 0   & 0 \\
c_1 & 0      & 0      & \dots     & 0       & 0   & c_1 \\
 \vdots &   \vdots     & \vdots &    & \vdots  &   \vdots  &  \vdots\\
	c_2   & c_1      & 0    & \dots   &  0    &  c_1  &  c_2 \\		
	c_1   & 0      & 0    & \dots     &  0    & 0   & c_1 \\		
	0   & 0      & c_1    & \dots    & c_1     & 0   & 0 \\	
		\end{pmatrix} 
		\]
		to which the following row 
		\[ 
	  \begin{pmatrix} 	
	0   & c_1    & c_2    & \dots   & c_2     & c_1 & 0 \\
		\end{pmatrix}   
	\]
	has to be added when $k = l+1$. 
	That matrix has the following properties 
			\begin{itemize}
			
			\item every column is repeated twice except the middle one when $l$ is odd  
			--more precisely, the $j$th and $(l-j+1)$th columns are identical for $1 \leq j \leq \lfloor l/2 \rfloor$--, and 
			
			\item every row is repeated twice except 
			(i) the first one when $k=l$, 
			(ii) the $(l/2+1)$th when $k=l$ is even, 
			(ii) the middle one when $k=l+1$ is odd 
			--more precisely, the $(i+s)$th and $(k-i+1)$th rows are identical for $1 \leq i \leq \lfloor k/2 \rfloor$,  
			and $s=0$ for $k = l+1$ and $s=1$ for $k = l$.   
			
			\end{itemize} 			
			
This concludes the recursion step, hence the proof. 
		\end{proof}
	
	A Matlab code that 
	generates the slack matrices of regular $n$-gons 
	and constructs the nonnegative factorizations described above for any $n$ 
	is available from 	
\begin{center}
\url{https://sites.google.com/site/exactnmf/regularngons}. 
\end{center}

\paragraph{Tightness of the Bound} 
It has to be pointed out that our inspiration for constructing the nonnegative factorizations used in Theorem~\ref{mainth} came from factorizations computed by our numerical solver~\cite{VGGT14} available on 
\url{https://sites.google.com/site/exactnmf/}.  

		Moreover, for $n$ up to $78$, the heuristic algorithm developed in~\cite{VGGT14} always found a factorization for the bound of Theorem~\ref{mainth} but never smaller. This suggests that our upper bound is tight, at least for small $n$.

\section{Conclusion} \label{conclu}

In this paper, we have first proposed a new lower bound for the rectangle covering number of the slack matrix of any $n$-gons, using a generalization of Sperner theorem; see Theorem~\ref{mainlow} and Corollary~\ref{maincor}. We hope that this idea will lead to new lower bound for other types of nonnegative matrices. 

Then, we proposed an algebraic proof for the upper bounds for the extension complexity of regular $n$-gons based on explicit nonnegative factorizations of the slack matrices of regular $n$-gons; see Theorem~\ref{mainth}.  
This bound slightly improves upon the previously best known upper bound from~\cite{FRT12} (our improvement essentially comes from improving the base case but we provided a new purely algebraic proof), and allows us to close the gap with the best known lower bound for several $n$-gons ($9 \leq n \leq 13$, $21 \leq n \leq 24$; see Figure~\ref{compabounds}). However, for most $n$-gons (precisely, for $n=14$, $17 \leq n \leq 20$, $25 \leq n \leq 30$ and $n \geq 33$), 
there is still a gap with the best known lower and upper bounds hence it is a direction for further research to improve these bounds to determine the extension complexity of these regular $n$-gons. 
Our numerical results suggest that the way to go would be to improve the lower bounds since our upper bound appears to be tight, 
at least for small $n$.

\section{Acknowledgement} We kinldy acknowledge the participants of the Dagstuhl seminar 15082 
on `Limitations of convex programming: lower bounds on extended formulations and factorization ranks' for insightful discussions, and we thank in particular the organizers, 
Hartmut Klauck, 
Troy Lee, 
Dirk Oliver Theis, and 
Rekha R. Thomas. We also thank the two anonymous reviewers for their insightful comments which helped improve the paper significantly.   
Finally, we thank Jo\~ao Gouveia for insightful discussions and for giving us the reference to the upper bounds for the boolean rank~\cite{BHJL86}. 

\small 

\bibliographystyle{spmpsci} 
\bibliography{Biography}

\newpage 

\appendix

\section{Proof for Theorem~\ref{mainlow}} \label{appA} 

The solution $k^* = \lfloor r/2 \rfloor$ and $z^* = 1$ is optimal for  
\[ 
\min_{k\geq 1, z \geq 1, k+z \leq r} k! (r-k)! 
+ (k+z)! (r-k-z)! - 2 k! z! (r-k-z)! . 
\] 

\begin{proof} 
Let us observe the following 
\begin{itemize} 
\item the first (resp.\@ second) term is decreasing when $k$ (resp.\@ $k+z$) gets closer to $r/2$.  

\item the last term is strictly increasing in $z$ hence being minimized in $z=1$. 

\item $f(k,z) = f( r-k-z , z)$.  (Note that this implies that, for $r$ even, $k^* = r/2 -1$ is also optimal.)  

\end{itemize}

 The first two observations imply that, at optimality, the case $z \geq 2$ and $k+z \geq \lfloor r/2 \rfloor+1$ is not possible, otherwise we would decrease the objective function by decreasing $z$. 
In other words, either $z^* = 1$ or $k+z \leq \lfloor r/2 \rfloor$. 



\noindent \textbf{Case 1: $z^* = 1$.} Since $f(k,1) = f(r-k-1 , 1)$, we can assume w.l.o.g.\@ that $k \geq \lfloor r/2 \rfloor$ since either $k$ or $r-k-1$ is larger than $\lfloor r/2 \rfloor$. Showing that 
$f(k,1)$ 
is increasing for $\lfloor r/2 \rfloor \leq k \leq r-1$, that is, that $f(k,1) \leq f(k+1,1)$ for $k+1 \leq r-1$ will prove the result: 
\[
k! (r-k)! + (k+1)! (r-k-1)! - 2 k! (r-k-1)! 
\leq 
(k+1)! (r-k-1)! + (k+2)! (r-k-2)! - 2 (k+1)! (r-k-2)!
\] 
\[
\iff 
\]
\[
k! (r-k)! - 2 k! (r-k-1)! 
\leq 
  (k+2)! (r-k-2)! - 2 (k+1)! (r-k-2)! . 
\] 
Dividing by $k!$ and $(r-k-2)!$, 
 \[
 (r-k) (r-k-1) - 2  (r-k-1)
\leq 
  (k+2)(k+1)  - 2 (k+1) 
\] 
which is equivalent to 
\[
r^2 - 3r + 2 \leq 2k(r-1). 
\]
Since $k \geq \lfloor r/2 \rfloor$, $2k \geq r-1$ hence the above inequality would be implied by 
\[
r^2 - 3r + 2 \leq (r-1)^2 = r^2 - 2r + 1 \iff r \geq 1.  
\]

\noindent \textbf{Case 2: $k+z \leq \lfloor r/2 \rfloor$.}  
We have $k' = r-k-z \geq \lfloor r/2 \rfloor$ hence we can reduce this case to the case $k \geq \lfloor r/2 \rfloor$ without loss of generality, since $f(k,z) = f( r-k-z , z)$. 
For $k \geq \lfloor r/2 \rfloor$, it is clear that $z^* = 1$ is optimal (since last tow terms increase with $z$ in that case) 
so that this case reduces to case~1 when $z^* = 1$. 
 
\end{proof}  

\newpage

\section{Code for the Nonnegative Factorization of Slack Matrices of Regular $n$-gons} \label{matcode}

\begin{lstlisting}
% Rank r nonnegative factorization of the slack matrix S_n=UV of the regular 
% n-gon generated by the function slack.m, r being equal to 
% 2k-1  for  2^{k-1}          <   n   <=  2^{k-1}+2^{k-2} , and 
% 2 k   for  2^{k-1}+2^{k-2}  <   n   <=  2^{k} . 
% 
% See A. Vandaele, N. Gillis and F. Glineur, 
% "On the Linear Extension Complexity of Regular n-gons", arXiv, 2015. 
% If you use the code, please cite the paper.
% See also https://sites.google.com/site/exactnmf/regularngons 
% 
% This version uses the matrix S as an input with 0(n^2) operations, 
% hence is computationally less efficient as factorization.m which 
% only requires O(n log(n)).  
% However, it is more intuitive to understand the construction and follows 
% the proof of the paper above more closely. 

function [U,V,R] = NonnegFactoRegnGon(S)

[m,n] = size(S); 
if n <= 4 % trivial factorization
    U = S; 
    V = eye(n); 
    R = S; 
else n > 4 % non-trivial factorizations
    % Step 1: Create the cross pattern of zeros removing a nonnegative
    % rank-two factor 
    [U,V,R] = offdiag_zeros(S); 
    % Step 2: Extract the upper left block that has the same nonnegative
    % rank as the full residual R (because of symmetry/redundancy)
    k1 = ceil(m/2);  
    if k1 == m/2 && m == n % When m is even and m == n
        k1 = k1+1;
    end
    k2 = ceil(n/2); 
    % Step 3: Factor the upper left block using recursion
    [Ur,Vr] = NonnegFactoRegnGon(S(1:k1,1:k2));   
    % Step 4: Put everything together using the symmetry
    r = size(Ur,2); 
    U = [U zeros(m,r)]; 
    V = [V; zeros(r,n)]; 
    % Factor V
    V(3:end,1:k2) = Vr; 
    for i = n : -1 : k2+1
        V(3:end,i) = V(3:end,n-i+1);
    end
    % Factor U 
    U(1:k1,3:end) = Ur; 
    if m == n % Case 1: R(k1,:)==R(k1+1,:), symmetry is 'perfect'
        p = 1;
    elseif m == n+1 % Case 2: R(k1-1,:)==R(k1+1,:), symmetry is shifted by one
        p = 0;
    end
    for i = 1 : m-k1
       U(k1+i,3:end) = U(k1-i+p,3:end);
    end
end
\end{lstlisting} 

\begin{lstlisting}
% Add zeros on off-diagonal entries of matrix S using a rank-two
% correction. The first rank-one factor puts zero entries on the lower left
% block, the second on the upper right block. 

function [U,V,R] = offdiag_zeros(S) 

[m,n] = size(S); 
U = zeros(m,2); 
V = zeros(2,n); 
k2 = floor(n/2); 
% Lower left block
if m == n  % zeros below the diagonal (starting from the lower left)
    k1 = ceil(n/2); 
    U(m,1) = S(m,1); 
elseif m == n+1  % zeros above the diagonal (starting from the lower left)
    k1 = floor(m/2); 
    U([m-1 m],1) = S([m-1 m],1); 
else
    error('The matrix should be n-by-n or n+1-by-n');
end 
V(1,1) = 1;  
for i = 2 : k2
    V(1,i)   = S(n-i+2,i) / U(n-i+2,1); 
    U(n-i+1,1) = S(n-i+1,i)/ V(1,i); 
end
% Upper right block: zeros below the diagonal
% (starting from the upper right)
V(2,n) = 1; U(1,2) = S(1,n); 
for i = 2 : k1
    U(i,2)    = S(i,n-i+2) /  V(2,n-i+2); 
    V(2,n-i+1)= S(i,n-i+1) / U(i,2); 
end 
% Residual with the pattern of zeros like a cross
R = S - U*V; 
\end{lstlisting} 

The code is available from \url{https://sites.google.com/site/exactnmf/regularngons}.  

\end{document}